\documentclass[10pt,notitlepage,reqno]{amsart}

\usepackage{verbatim}
\usepackage{amssymb,amsmath,amsthm,mathrsfs,mathtools}
\usepackage{amscd}
\usepackage[breaklinks=true]{hyperref}
\usepackage{url}
\usepackage{units}
\usepackage[usenames,dvipsnames]{xcolor}
\usepackage{graphicx}
\usepackage[all]{xy}
\usepackage{amsmath}
\usepackage{cleveref}
\usepackage{xy}
\hypersetup{
  colorlinks,
  linkcolor	=	Gray,
  urlcolor 	= Gray,
  citecolor	=	Gray,
  pdfauthor={Mirco Richter},
  pdfkeywords={Lie, infinity, Rinehart, Schouten, symplectic, multisymplectic},
  pdftitle={Higher Symplectic Structures on torsionless Lie Rinehart pairs},
  pdfsubject={Lie, Rinehart, symplectic},
  pdfpagemode=UseNone
}

\let\originalleft\left
\let\originalright\right
\renewcommand{\left}{\mathopen{}\mathclose\bgroup\originalleft}
\renewcommand{\right}{\aftergroup\egroup\originalright}

\newcommand{\bwedge}{\raisebox{0.3ex}{${\scriptstyle\bigwedge\,}$}}

\newcommand{\boplus}{\raisebox{0.2ex}{${\textstyle \bigoplus}$}}

\newcommand{\smwedge}{{\scriptstyle \;\wedge\;}}

\newcommand{\ordinal}[1]{[$\;\!#1\,$]}
\newcommand{\R}{\mathbb{R}}
\newcommand{\Z}{\mathbb{Z}}
\newcommand{\N}{\mathbb{N}}

\newcommand{\NN}{\mathbb{N}_0}

\newcommand{\I}{\infty}

\newcommand{\BIGOP}[1]{\mathop{\mathchoice
{\raise-0.22em\hbox{\huge $#1$}}
{\raise-0.05em\hbox{\Large $#1$}}{\hbox{\large $#1$}}{#1}}}

\def\clap#1{\hbox to 0pt{\hss#1\hss}}

\theoremstyle{plain}
\newtheorem{theorem}{Theorem}[section]
\newtheorem{proposition}[theorem]{Proposition}

\newtheorem{corollary}[theorem]{Corollary}

\newtheorem{definition}[theorem]{Definition}

\theoremstyle{remark}
\newtheorem*{remark}{Remark}
\newtheorem{example}{Example}

\title{Higher Symplectic Structure on torsionless Lie Rinehart pairs}
\author{Mirco Richter }
\date {\today}
\thanks{email: \href{mailto:mirco.richter@email.de}{\tt mirco.richter@email.de}}

\begin{document}

\maketitle
\begin{abstract}
We define an n-plectic structure as a commutative and torsionless
Lie Rinehart pair, together with a distinguished cocycle 
from its Chevalley-Eilenberg complex. 

This 'n-plectic cocycle' gives rise to an extension of the 
Chevalley-Eilenberg complex by so called symplectic tensors.
The cohomology of this extension generalizes Hamiltonian 
functions and vector fields to tensors and cotensors in a 
range of degrees, up to certain coboundaries and
has the structure of a Lie $\I$-algebra.

Finally we show, that momentum maps appear in this context just
as weak Lie $\infty$-morphisms.
\end{abstract}

\section{Introduction}
We define n-plectic structures as commutative 
and torsionless Lie Rinehart pairs, together with a distinguished 
$(n+1)$-cocycle from their Chevalley-Eilenberg complexes. 

By torsionless we mean, that the Lie algebra partner of the 
pair is a torsionless module with respect to its commutative partner. 
The natural pairing between tensors and cotensors is then non-degenerate
and we can define general Cartan calculus.

In contrast to symplectic geometry, where the symplectic form has to
be non-degenerate on vectors, we do not impose additional properties
on the $n$-plectic cotensor, other than being a cocycle. Since we deal 
with tensors of arbitrary degree almost all n-plectic cocycles will 
be degenerate to some extend on higher tensors and therefore 
we drop any distinction between degenerate and non-degenerate right 
from the beginning.

The theory in this paper appears as a refinement of the work from Rogers on
$n$-plectic manifolds as given in \cite{CR} and the
Lie $\I$-algebra of Lie Rinehart pairs as given by me in \cite{MR}. 
In fact the papers \cite{FRZ} and \cite{FRS} have been a important inspiration.
 
The basic generalization however is, that we do not restrict to Hamiltonian 
(or symplectic) vectors, but consider appropriate tensors of 
arbitrary degree instead. This is possible, since the $n$-plectic structure 
behaves well with respect to the Lie $\I$-algebra of exterior tensors 
and we can define higher Poisson brackets as some kind of 
Lie $\I$-algebra extension.

A main concerns of this work is to convince the reader, that
exterior (co)tensors of higher degrees are a central part
of the theory, not only as a need from physics but for pure mathematical
reasons. In fact this idea has a long history, as 
it appears first in the work of Kanatchikov
(see \cite{K1}, \cite{K2}) and was then pushed further by Forger, Paufler and R\"omer
as in \cite{FPR} and the references therin.
 
In addition we show, that n-plectic geometry can 
be defined verbatim on arbitrary torsionless Lie Rinehart pairs without
much additional afford.
\section{Lie Rinehart pairs}We start with a short introduction
to Lie Rinehart pairs, which have the additional property, 
that the natural map from the Lie partner into its double dual module is 
injective. We call them torsionless or semi-reflexive. In analogy 
to the prototypical Lie Rinehart pair of smooth functions and vector 
fields, these pairs allow for a general definition of 
exterior derivative, contraction and Lie derivative.
 
In what follows $\mathfrak{g}$ will always be a real Lie algebra, that is a $\R$-vector space together with an antisymmetric, bilinear map,
\begin{equation}
[\cdot,\cdot]: \mathfrak{g}\times\mathfrak{g}\to\mathfrak{g}
\end{equation}
called \textbf{Lie bracket}, such that for any three vector $x_1$, $x_2$ and $x_3\in\mathfrak{g}$ the \textbf{Jacobi identity}
$[x_1,[x_2,x_3]]+[x_2,[x_3,x_1]]+[x_3,[x_1,x_2]]=0$ is satisfied.

In addition $A$ will always be a real
associative and commutative algebra with unit, 
that is a $\R$-vector space together with an associative and commutative, bilinear map
\begin{equation}
\cdot : A \times A \to A
\end{equation}
called multiplication and a unit $1_A\in A$.
According to a better readable text, 
we frequently suppress the symbol of the multiplication in $A$ and just
write $ab$ instead of $a\cdot b$.

Moreover $Der(A)$ will be the Lie algebra of derivations of $A$, 
that is the vector space of linear endomorphisms of $A$, 
with $D(ab)=D(a)b+aD(b)$ and Lie bracket
$[D,D'](a):=D(D'(a))-D'(D(a))$ for any $a,b\in A$ and $D,D'\in Der(A)$.

Before we get to Lie Rinehart pairs, it is handy to define Lie algebra modules first:
\begin{definition}[Lie algebra module]
Let $\mathfrak{g}$ be a real Lie algebra, $A$ an $\R$-algebra and
$D:\mathfrak{g}\to Der(A)$ a Lie algebra morphism. Then $A$ is called
a \textbf{Lie algebra module} (or just $\mathfrak{g}$-module) and $D$
is called the $\mathfrak{g}$\textbf{-scalar multiplication}.
\end{definition}
Now a \textit{Lie Rinehart pair} is nothing but a Lie algebra 
and an associative algebra, 
each of them being a module with respect to the other, such that 
a particular compatibility equation of their products is satisfied: 
\begin{definition}[Lie Rinehart Pair] Let $A$ be an associative and 
commutative algebra with unit, $\mathfrak{g}$ a Lie algebra and 
$\cdot_A: A\times \mathfrak{g} \to \mathfrak{g}$ as well as
$D: \mathfrak{g}\to Der(A)\;;\; x\mapsto D_x$ maps, such that 
$A$ is a $\mathfrak{g}$-module with $\mathfrak{g}$-scalar multiplication
$D$, the vector space $\mathfrak{g}$ is
an $A$-module with $A$-scalar multiplication $\cdot_A$ and 
the \textbf{Leibniz rule}
\begin{equation}
[x,a\cdot_A y] = D_x(a)\cdot_A y + a \cdot_A[x,y]
\end{equation}
is satisfied for any $x,y\in\mathfrak{g}$ and $a\in A$. Then
$\left(A,\mathfrak{g}\right)$ is called a \textbf{Lie Rinehart pair}. 
\end{definition}
Maybe the most prominent example is provided by smooth functions and vector fields on a 
differentiable manifold:
\begin{example}
Let $M$ be a differentiable manifold, $C^\I(M)$ the algebra of smooth, 
real valued functions and $\mathfrak{X}(M)$ the Lie algebra of vector
fields on $M$. $\mathfrak{X}(M)$ is a $C^\I(M)$-module and 
vector fields acts as derivations on smooth functions, that is the
map
$$D:\mathfrak{X}(M)\times C^\I(M) \to C^\I(M)\;;\;
 (X,f) \mapsto D_X(f):= X(f)
$$
satisfies the equation
$D_X(fg)=D_X(f)g + fD_X(g)$. Moreover the Leibniz rule
$[X,fY]=D_X(f)Y+f[X,Y]$ holds and it follows that
$(C^\I(M),\mathfrak{X}(M))$ is a Lie Rinehart pair.
\end{example} 
Morphisms of Lie Rinehart pairs are pairs
of appropriate algebra maps, which interact properly with
respect to the additional module structures \cite{GR}:
\begin{definition}[Lie Rinehart Morphism]
Let $(A,\mathfrak{g})$ and $(B,\mathfrak{h})$ be two Lie Rinehart pairs.
A \textbf{morphism of Lie Rinehart pairs} is a pair of maps $(f,g)$, such
that $f:A\to B$ is a morphism of associative and commutative, real algebras
with unit, $g:\mathfrak{g}\to\mathfrak{h}$ is a morphism of Lie
algebras and the equations
\begin{equation}\label{LR-morph}
\begin{array}{ccc}
g(a\cdot_A x)=f(a)\cdot_B g(x)
& and &
f(D_x(a))=D_{g(x)}(f(a))
\end{array}
\end{equation} 
are satisfied for any $a\in A$ and $x\in \mathfrak{g}$.  
\end{definition}
The Lie structure of $\mathfrak{g}$
can be extended into a graded\footnote{The term 'graded' will 
always mean $\Z$-graded. See the appendix for an introduction to $\Z$-graded 
vector spaces.}
Lie algebra on the
direct sum of the partners, concentrated in degrees zero and one. 
This differs from \cite{GR}, where no grading was
considered.
\begin{definition}[Associated Lie algebra]
Let $(A,\mathfrak{g})$ be a Lie Rinehart pair. Its 
\textbf{associated} (graded) Lie algebra is the graded direct
sum $A\oplus\mathfrak{g}$, where $A$ is concentrated in degree zero, 
$\mathfrak{g}$ is concentrated
in degree one and the Lie bracket is defined by
\begin{equation}
\begin{array}{crcl}
[\cdot,\cdot]: & A\oplus\mathfrak{g} \times A\oplus\mathfrak{g} &\to&
 A\oplus\mathfrak{g}\\
 & \left((a,x),(b,y)\right) &\mapsto& \left(D_x(a)+D_y(b),[x,y]\right)\;.
\end{array}
\end{equation}
\end{definition} 
In particular this means, that we can talk about Lie $\I$-morphisms in 
the context of Lie Rinehart pairs.

To define torsionless Lie Rinehart pairs, let 
$\mathfrak{g}^\vee_A:=Hom_{A}(\mathfrak{g},A)$ be the dual of 
the $A$-module $\mathfrak{g}$ and $\mathfrak{g}^{\vee\vee}$ its appropriate double
dual. 
\begin{definition}[Torsionless Lie Rinehart pair]
Let $(A,\mathfrak{g})$ be a Lie Rinehart pair, such that the natural map
$\mathfrak{g}\to \mathfrak{g}^{\vee\vee}\;;\;
x\mapsto \left(\mathfrak{g}^\vee \to A\;;\; f\mapsto f(x)\right)$ is 
injective. Then $(A,\mathfrak{g})$ is called \textbf{torsionless} or 
\textbf{semi-reflexive}.
\end{definition}
Torsionless Lie Rinehart pairs gives rise to a non degenerate pairing between 
vectors and covectors, which in turn allows to consider Cartan calculus. 
To see that we define $\otimes^0_A\mathfrak{g}:=A$ 
as well as $\otimes^0_A\mathfrak{g}^\vee_A:=A$ and 
write $\otimes^n_A\mathfrak{g}$ as well as 
$\otimes^n_A\mathfrak{g}^\vee_A$ for the 
appropriate $n$-fold $A$-tensor products of the $A$-modules 
$\mathfrak{g}$ and $\mathfrak{g}^\vee$. Since $A$ is commutative, 
$\otimes^n_A\mathfrak{g}$ and $\otimes^n_A\mathfrak{g}_A^\vee$ 
are $A$-modules.
\newpage
\begin{definition}[Exterior Algebras]\label{tensor_grading} 
Let $(A,\mathfrak{g})$ be a Lie Rinehart pair and 
$n\in\Z$. For $n<0$ define $E^n(A,\mathfrak{g})=\{0\}$ and for $n\geq 0$
let $E^n(A,\mathfrak{g}):=\otimes^n_A\mathfrak{g}/J^n$ be
the quotient $A$-module of the $n$-th tensor product and the submodule $J^n$, 
spanned by all $x_1\otimes\cdots\otimes x_n$ with 
$x_i = x_j$ for some $i = j$. Then the direct sum 
\begin{equation}
E(A,\mathfrak{g}):=\textstyle\bigoplus_{n\in\Z} E^n(A,\mathfrak{g})
\end{equation}
together with the quotient 
$\smwedge: E(A,\mathfrak{g}) \times E(A,\mathfrak{g}) \mapsto
E(A,\mathfrak{g})\;;\; (x,y)\mapsto x\smwedge y$ of the
$A$-tensor multiplication, is called the \textbf{exterior tensor algebra} of 
$(A,\mathfrak{g})$ and the product is called the \textbf{exterior tensor product}. 

Dual if $n<0$ define $E_{-n}(A,\mathfrak{g}^\vee)=\{0\}$ and if $n\geq 0$ let 
 $E_{-n}(A,\mathfrak{g}^\vee):=\otimes^n_A\mathfrak{g}^\vee_A/J^n$ be
the quotient $A$-module of the $n$-th tensor product of the dual module
and the submodule $J^n$, spanned by all cotensors $x^1\otimes\cdots\otimes x^n$ with 
$x^i = x^j$ for some $i = j$. Then the direct sum 
\begin{equation}
E(A,\mathfrak{g}^\vee)=\textstyle\bigoplus_{n\in\Z} E_{-n}(A,\mathfrak{g}^\vee)
\end{equation}
together with the quotient 
$\smwedge: E(A,\mathfrak{g}^\vee) \times E(A,\mathfrak{g}^\vee) \mapsto
E(A,\mathfrak{g}^\vee)\;;\; (x,y)\mapsto x\smwedge y$ of the
$A$-tensor multiplication, is called the \textbf{exterior cotensor algebra} of 
$(A,\mathfrak{g})$ and the product is called the \textbf{exterior cotensor product}. 
\end{definition}
\begin{example}
If $(C^\I(M),\mathfrak{X}(M))$ is the Lie Rinehart pair of smooth 
functions and vector fields, the exterior cotensor algebra 
is the algebra of \textbf{differential forms} $\Omega(M)$ and
the exterior tensor algebra is the algebra of \textbf{multivector fields}
$\bwedge\mathfrak{X}(M)$.
\end{example}
A tensor $x\in E^n(A,\mathfrak{g})$ is called homogeneous of tensor
degree $n$ (written as $|x|=n$) and a cotensor 
$f\in E_{-n}(A,\mathfrak{g}^\vee)$ is
called homogeneous of tensor degree $-n$ (written as $|f|=-n$) for any
$n\in\N$. The appropriate grading 
is called the \textbf{tensor grading}. Note that tensors of negative degrees 
as well as cotensors of positive degrees are zero and
$E_0(A,\mathfrak{g}^\vee)\simeq A\simeq E^0(A,\mathfrak{g})$ as well as 
$E^1(A,\mathfrak{g})\simeq \mathfrak{g}$ and 
$E_{-1}(A,\mathfrak{g}^\vee)\simeq \mathfrak{g}_A^\vee$.

This grading is a natural choice in the context of $\Z$-graded modules,
since the usual pairing of tensors and cotensors, appears as a 
graded map, homogeneous of degree zero:
\begin{definition}[Pairing]Let $(A,\mathfrak{g})$
be a Lie Rinehart pair. The $A$-bilinear map
\begin{equation}
\langle\cdot,\cdot\rangle:E(A,\mathfrak{g}^\vee)\times E(A,\mathfrak{g})\to A,
\end{equation}
defined by $\langle a,b\rangle=ab$ on scalars $a,b\in A$, by
$\langle f,x\rangle=0$ on cotensors $f\in E_p(A,\mathfrak{g}^\vee)$ and
tensors $x\in E^q(A,\mathfrak{g})$, such that $p+q\neq 0$ and by
$$
\langle f^1\smwedge\cdots\smwedge f^n, x_1\smwedge\cdots\smwedge x_n\rangle=
det(f^j(x_i))
$$
on simple cotensors $f^1\smwedge\cdots\smwedge f^n\in E_{-n}(A,\mathfrak{g})$
and tensors $x_1\smwedge\cdots\smwedge x_n\in E^n(A,\mathfrak{g})$
and then extended to all of 
$E_{-n}(A,\mathfrak{g}^\vee)\times E^n(A,\mathfrak{g})$ 
by $A$-linearity, is called the \textbf{natural pairing}.
\end{definition}
\begin{proposition}Let $(A,\mathfrak{g})$ be a torsionless Lie Rinehart pair. Then
the natural pairing is non degenerate.
\end{proposition}
\begin{proof}
This follows since the natural map $\mathfrak{g}\to \mathfrak{g}^{\vee\vee}$ 
is injective.
\end{proof}
With a non degenerate pairing the contraction of a cotensor along a tensor
can be defined consistently.
\begin{definition}[Contraction]
Let $(A,\mathfrak{g})$ be a torsionless Lie Rinehart pair and
$x\in E^k(A,\mathfrak{g})$ a tensor homogeneous of degree $k$. Then
the map
\begin{equation}
i_x:E_{-n}(A,\mathfrak{g}^\vee)\to E_{k-n}(A,\mathfrak{g}^\vee)
\end{equation}
uniquely defined for any $n\in\NN$ and homogeneous cotensor 
$f\in E_{-n}(A,\mathfrak{g}^\vee)$ by the equation
$$
\langle i_xf, y \rangle = \langle f, x\smwedge y\rangle
$$
for all $y\in E^{n-k}(A,\mathfrak{g})$ and then extended to all of 
$E(A,\mathfrak{g}^\vee)$ by additivity, is called the 
\textbf{contraction along }$x$. 
\end{definition}
Since the natural pairing is non degenerate, the contraction is uniquely defined
by this equation.

The exterior tensor power of a Lie Rinehart pair has the structure of a 
Lie $\I$-algebra \cite{MR}, with higher operators defined in terms of the exterior 
product and the Schouten-Nijenhuis bracket.
\begin{definition}[Higher Lie Brackets]
Let $(A,\mathfrak{g})$ be a Lie
Rinehart pair with exterior algebra $E(A,\mathfrak{g})$ and 
$[\cdot,\cdot]_S$ the antisymmetric Schouten-Nijenhuis bracket. Then the tensor
\textbf{Lie k-bracket}
\begin{equation}\label{Higher_Lie_bracket}
[\cdot,\cdots,\cdot]_k: E(A,\mathfrak{g}) \times \cdots \times
 E(A,\mathfrak{g}) \to E(A,\mathfrak{g})
\end{equation}
is defined for any integer $k \geq 2$ and homogeneous tensors
$x_1,\ldots,x_k\in E(A,\mathfrak{g})$ by
\begin{multline*}
[x_1,\ldots,x_k]_k:=\\ \textstyle\sum_{s\in Sh(2,k-2)}e(s;x_1,\ldots,x_k)e(x_{s(1)})\;
x_{s(k)}\smwedge\cdots\smwedge x_{s(3)}\smwedge[x_{s(2)},x_{s(1)}]_S
\end{multline*}
and is then extended to all of $E(A,\mathfrak{g})$ by $A$-additivity.
\end{definition}
These higher tensor brackets are graded symmetric and homogeneous of degree
$-1$ with respect to the tensor grading.
The unary bracket $[\,\cdot\,]_1$ has to be the zero operator in
general \cite{MR} and
$\left(E(A,\mathfrak{g}),[\cdot,\ldots,\cdot]_{k\in\N}\right)$
is a Lie $\I$-algebra, concentrated in non negative (tensor) degrees. 

In what follows we will always write $[\cdot,\ldots,\cdot]_k$ for the 
graded symmetric higher tensor brackets as well as $[\cdot,\cdot]_S$ for the
common Schouten-Nijenhuis bracket. 

The natural inclusion 
$A\oplus\mathfrak{g}\hookrightarrow E(A,\mathfrak{g})\;;\;
(a,x)\mapsto(a,x)$ 
is not a morphism of Lie $\I$-algebras, 
since such a strict morphism has to commute with all higher brackets, 
but these brackets are zero on $A\oplus\mathfrak{g}$.
Instead the inclusion now comes as a weak morphism of Lie $\I$-algebras: 
\begin{definition}Let $(A,\mathfrak{g})$ be a Lie Rinehart pair with 
associated graded Lie algebra $A\oplus\mathfrak{g}$,
exterior tensor algebra $E(A,\mathfrak{g})$ and 
\begin{equation}
\begin{array}{crcl}
i_k: & A\oplus\mathfrak{g}\times \cdots \times A\oplus\mathfrak{g} &\to &
 E(A,\mathfrak{g})\\
     &   (x_1,\ldots,x_k) & \mapsto &
 (-1)^{k-1}(k-1)\,!\cdot x_k\smwedge\cdots\smwedge x_1
\end{array}
\end{equation}
for any $k\in \N$. Then the sequence $i_\I:=(i_k)_{k\in\N}$ is called the
\textbf{natural inclusion} of the Lie Rinehart pair into its exterior 
tensor Lie $\I$-algebra.
\end{definition} 
On the other side, the exterior \textit{cotensor} algebra has
the structure of a differential graded algebra, at least if the
Lie Rinehart pair is torsionless. This generalizes the well
known De Rham Complex of differential forms to the exterior
cotensor algebra of such Lie Rinehart
pairs:
\begin{definition}[Exterior derivative]
Let $(A,\mathfrak{g})$ be a torsionless Lie Rinehart pair with exterior 
cotensor algebra $E(A,\mathfrak{g}^\vee)$. Then the \textbf{exterior derivative}
$$
d:E(A,\mathfrak{g}^\vee)\to E(A,\mathfrak{g}^\vee)
$$
is defined for a homogeneous cotensor $f\in E_{-k}(A,\mathfrak{g}^\vee)$ 
and any $x_0,\ldots,x_k\in \mathfrak{g}$ by the equation
\begin{multline*}
df(x_0,\ldots,x_k) 
= \textstyle\sum_j (-1)^{j} 
 D_{x_j}(f(x_0, \ldots, \widehat{x}_j, \ldots,x_k))\\
+\textstyle\sum_{i<j}
 (-1)^{i+j}f([x_i, x_j], x_0, \ldots, \widehat{x}_i, \ldots, 
  \widehat{x}_j, \ldots, x_k)
\end{multline*}
and is then extended to all of $E(A,\mathfrak{g}^\vee)$ by $A$-linearity. 

The complex $\left(E(A,\mathfrak{g}^\vee),d\right)$ is called the 
\textbf{Chevalley-Eilenberg complex} of the Lie Rinehart pair.
\end{definition}
Since the natural pairing is non degenerate, the exterior derivative is
uniquely defined by this equation and is moreover a graded map, 
homogeneous of degree $-1$, with $d^2=0$. 
In particular the Chevalley-Eilenberg complex is a
cochain complex with respect to the tensor grading.

With the contraction along tensors and the exterior
derivative, we can now define Lie derivatives along tensors in terms of the 
graded Cartan formula: 
\begin{definition}Let $(A,\mathfrak{g})$ be a torsionless Lie Rinehart pair and
$x\in E(A,\mathfrak{g})$ a homogeneous tensor. Then the map
\begin{equation}\label{Cartans_formula}
L_x: E(A,\mathfrak{g}^\vee)\to E(A,\mathfrak{g}^\vee)\;;\;
f\mapsto di_xf - (-1)^{|x|}i_xdf
\end{equation}
is called the \textbf{Lie derivative} along $x$.
\end{definition}
The following proposition summarizes basic computation rules as we need them in
what follows:
\begin{proposition}Let $(A,\mathfrak{g})$ be a torsionless Lie Rinehart 
pair, $[\cdot,\cdot]_S$ the graded antisymmetric Schouten-Nijenhuis bracket, 
$x,y\in E(A,\mathfrak{g})$ homogeneous tensors and $f\in E(A,\mathfrak{g}^\vee)$
a cotensor. Then
\begin{equation}\label{multi_rules}
\begin{array}{lcl}
dL_x f &=&(-1)^{|x|-1}L_x df \\
i_{[x,y]_S} f &=&(-1)^{(|x|-1)|y|}L_x i_y\, f - i_y L_x f\\
L_{[x,y]_S} f &=& (-1)^{(|x|-1)(|y|-1)}L_x L_x f - L_y L_x f\\
L_{x\wedge y} f &=& (-1)^{|y|} i_y L_x f + L_y i_x f 
\end{array}
\end{equation}
\end{proposition}
\begin{proof} A proof is given in \cite{FPR} in the context of multivector fields
and differential forms, but can be carried over verbatim into our general setting.
\end{proof}
\section{Higher Symplectic Geometry}
\subsection{N-plectic structures}
We define higher symplectic structures as 
torsionless Lie Rinehart pairs together with a 
distinguished cocycle from their Chevalley-Eilenberg complex.
We look at the 'infinitesimal invariants' of such a cocycle
and show that they are a Lie $\I$-algebra with respect to the
higher tensor brackets. 

\begin{definition}Let $(A,\mathfrak{g})$ be a torsionless 
(semi-reflexive) Lie Rinehart pair and 
$\omega\in E_{-(n+1)}(A,\mathfrak{g}^\vee)$ a cocycle of tensor degree $-(n+1)$
for some $n\in\N$. Then
$(A,\mathfrak{g},\omega)$ is called an \textbf{n-symplectic} 
(or just n-plectic) \textbf{structure} and $\omega$ is called its 
\textbf{n-plectic cocycle}.
\end{definition}
We do not distinguish between $n$-plectic cocycles that are degenerate on
vectors and those that are not. The fundamental pairing between 
functions and vector fields generalizes to 
tensors and cotensors in a range of degrees and almost all 
$n$-plectic cocycles have a non trivial kernel on higher tensors, 
whether they are degenerate on vectors or not.
Unique pairings are an exception in general $n$-plectic structures.
\begin{example}Any symplectic manifold $(M,\omega)$ is in particular a 
$1$-plectic structure $(C^\I(M),\mathfrak{X}(M),\omega)$ on the
Lie Rinehart pair of smooth functions and vector fields.
\end{example}
\begin{example}[Trivial $n$-plectic structure]Let 
$(A,\mathfrak{g})$ be a torsionless Lie Rinehart pair and 
$\omega\in E_{-(n+1)}(A,\mathfrak{g}^\vee)$ the zero cotensor. 
Then $(A,\mathfrak{g},\omega)$ is called the 
\textbf{trivial} $n$-plectic structure.
\end{example}
If the Lie derivation of a symplectic form along a vector vanishes, 
the vector is called symplectic. In a general $n$-plectic setting however, 
there is no reason to restrict to vectors, but to ask for 
\textit{all} infinitesimal invariants of an $n$-plectic cocycle:  
\begin{definition}[Symplectic and Hamiltonian tensors]
Let $(A,\mathfrak{g},\omega)$ be an $n$-plectic structure and
$x\in E(A,\mathfrak{g})$ an exterior tensor (of arbitrary tensor degree).
We say that $x$ is a \textbf{symplectic tensor} if the contraction of 
$\omega$ along $x$ is again a cocycle, that is if 
\begin{equation}\label{symplectic_1}
di_x\omega=0\;.
\end{equation}
In addition we say that $x$ is a \textbf{Hamiltonian tensor}, if
the contraction of $\omega$ along $x$ is a coboundary, i.e. if 
there is an exterior cotensor $f\in E(A,\mathfrak{g}^\vee)$ such that
\begin{equation}\label{symplectic_tensor_2}
df=i_x\omega\;.
\end{equation}
\end{definition}
We write $S\overline{ym}(A,\mathfrak{g},\omega)$ for the set of all symplectic tensors. Since
$d\omega=0$, an exterior tensor $x$ is symplectic, if and only if $L_x\omega=0$.
\begin{remark}
Any Hamiltonian tensor is clearly a symplectic tensor, but the converse
depends on the de-Rham cohomology of the Chevalley-Eilenberg complex
$E(A,\mathfrak{g}^\vee)$. However in sharp contrast to symplectic geometry, the
first cohomology group isn't sufficient anymore, since $i_x\omega$ now can have
a tensor degree other then $-1$.
\end{remark}
The following proposition is one of the central observations in this work
and a crucial technical detail. It links 
the Lie $\I$-algebra of exterior tensors to the cochain complex of
exterior cotensors.
\begin{proposition}
Let $(A,\mathfrak{g},\omega)$ be an $n$-plectic structure and
$[\cdot,\ldots,\cdot]_k$ the tensor Lie $k$-bracket on 
$E(A,\mathfrak{g})$. Then
\begin{equation}\label{fundamental_pairing_1}
i_{[x_1,\ldots,x_k]_k}\omega = di_{x_k\wedge\cdots\wedge x_1}\omega
\end{equation}
for any $k\in\N$ with $k\geq 2$ and symplectic tensors 
$x_1,\ldots,x_k\in S\overline{ym}(A,\mathfrak{g},\omega)$.
\end{proposition}
\begin{proof}We proof this by induction on $k$. For $k=2$ it is the
well known link between the Poisson bracket of functions and the 
Lie bracket of symplectic vector field, but generalized to symplectic tensors of 
arbitrary degree:
\begin{align*}
i_{[x_1,x_2]_2}\omega
&=e(x_1)i_{[x_2,x_1]_S}\omega
 =e(x_1,x_2)L_{x_2}i_{x_1}\omega\\
&=e(x_1,x_2)di_{x_2}i_{x_1}\omega
 =di_{x_2\wedge x_1}\;.
\end{align*}
For the induction step we assume 
$i_{[x_1,\ldots,x_k]_k}\omega=di_{x_k\wedge\cdots\wedge x_1}\omega$ for
some $k\geq 2$ and use
the definition of the higher tensor brackets (\ref{Higher_Lie_bracket}) 
as well as their graded symmetry, to 
transform the left side of equation (\ref{fundamental_pairing_1}) according to:
\begin{multline*}
i_{[x_{1},\ldots,x_{k+1}]_{k+1}}\omega=\\	
\textstyle\sum_{s\in Sh(2,k-1)}e(s;x_{1},\ldots,x_{k+1})e(x_{s(1)})
 i_{x_{s(k+1)}\wedge\cdots\wedge x_{s(3)}\wedge[x_{s(2)},x_{s(1)}]_S}\omega=\\	
\textstyle\frac{1}{2(k-1)!}\sum_{s\in S_{k+1}}e(s;x_{1},\ldots,x_{k+1})e(x_{s(1)})
 i_{[x_{s(2)},x_{s(1)}]_S}i_{x_{s(k+1)}\wedge\cdots\wedge x_{s(3)}}\omega\;.
\end{multline*}
Then we apply (\ref{multi_rules}) to the right side of the last expression. 
After simplification and reindexing this leads to
\begin{multline*}
\textstyle\frac{1}{2(k-1)!}\sum_{s\in S_{k+1}}e(s;x_{1},\ldots,x_{k+1})
 L_{x_{s(1)}}i_{x_{s(k+1)}\wedge\cdots\wedge x_{s(2)}}\omega\\	
 -\textstyle\frac{1}{2(k-1)!}\sum_{s\in S_{k+1}}e(s;x_{1},\ldots,x_{k+1})e(x_{s(1)})
   i_{x_{s(1)}}L_{x_{s(2)}}i_{x_{s(k+1)}\wedge\cdots\wedge x_{s(3)}}\omega\;.
\end{multline*}
Now we rewrite the Lie derivation operator in the previous
expression using Cartans formula (\ref{Cartans_formula}), which gives 
\begin{multline*}
\textstyle\frac{1}{2(k-1)!}\sum_{s\in S_{k+1}}e(s;x_{1},\ldots,x_{k+1})
 di_{x_{s(k+1)}\wedge\cdots\wedge x_{s(1)}}\omega\\
-\textstyle\frac{1}{(k-1)!}\sum_{s\in S_{k+1}}e(s;x_{1},\ldots,x_{k+1})e(x_{s(1)})
 i_{x_{s(1)}}di_{x_{s(k+1)}\wedge\cdots\wedge x_{s(2)}}\omega\\
+\textstyle\frac{1}{2(k-1)!}\sum_{s\in S_{k+1}}e(s;x_{1},\ldots,x_{k+1})
 e(x_{s(1)})e(x_{s(2)})i_{x_{s(2)}\wedge x_{s(1)}}
  di_{x_{s(k+1)}\wedge\cdots\wedge x_{s(3)}}\omega\;.
\end{multline*}
Since any $x_j$ is symplectic, we have $dx_j\omega=0$ and it follows, that 
the last sum in this expression vanishes for $k=2$. Using the graded 
symmetry of the exterior tensor product, we rewrite the first sum of the previous
expression according to
\begin{multline*}
\textstyle\frac{(k+1)!}{2(k-1)!}di_{x_{k+1}\wedge\cdots\wedge x_{1}}\omega\\	
-\textstyle\frac{1}{(k-1)!}\sum_{s\in S_{k+1}}e(s;x_{1},\ldots,x_{k+1})e(x_{s(1)})
 i_{x_{s(1)}}di_{x_{s(k+1)}\wedge\cdots\wedge x_{s(2)}}\omega\\
  +\textstyle\frac{1}{2(k-1)!}\sum_{s\in S_{k+1}}e(s;x_{1},\ldots,x_{k+1})
   e(x_{s(1)})e(x_{s(2)})
    i_{x_{s(2)}\wedge x_{s(1)}}di_{x_{s(k+1)}\wedge\cdots\wedge x_{s(3)}}\omega
\end{multline*}
and apply the induction hypothesis on this expression. Recall that if $k=2$, the
last sum vanishes, so this is well defined: 
\begin{multline*}
\textstyle\frac{k(k+1)}{2}di_{x_{k+1}\wedge\cdots\wedge x_{1}}\omega\\	
-\textstyle\frac{1}{(k-1)!}\sum_{s\in S_{k+1}}e(s;x_{1},\ldots,x_{k+1})e(x_{s(1)})
  i_{x_{s(1)}}i_{[x_{s(2)},\ldots,x_{s(k+1)}]_{k}}\omega\\	
+\textstyle\frac{1}{2(k-1)!}\sum_{s\in S_{k+1}}
 e(s;x_{1},\ldots,x_{k+1})e(x_{s(1)})e(x_{s(2)})
 i_{x_{s(2)}\wedge x_{s(1)}}i_{[x_{s(3)},\ldots,x_{s(k+1)}]_{k-1}}\omega
\end{multline*}
In the next step, we substitute the higher tensor brackets again by 
their definition (\ref{Higher_Lie_bracket}). 
(According to a better readable text we write $Sh_S(p,q)$ for 
the set of $(p,q)$-shuffle permutations defined explicit on the finite
set $S$.) This transforms the previous expression into
\begin{multline*}
\textstyle\frac{k(k+1)}{2}di_{x_{k+1}\wedge\cdots\wedge x_{1}}\omega\\
-\textstyle\frac{1}{(k-1)!}\sum_{s\in S_{k+1}}
 \sum_{t\in Sh_{\{s(2),\ldots,s(k+1)\}}(2,k-2)}
  e(s;x_{1},\ldots,x_{k+1})\,\cdot\\
   \cdot e(t;x_{s(2)},\ldots,x_{s(k+1)})
   e(x_{s(1)})e(x_{ts(2)})i_{x_{s(1)}}
    i_{x_{ts(k+1)}\wedge\cdots\wedge x_{ts(4)}\wedge[x_{ts(3)},x_{ts(2)}]_S}\omega
\end{multline*}
\begin{multline*}	
+\textstyle\frac{1}{2(k-1)!}\sum_{s\in S_{k+1}}
 \sum_{t\in Sh_{\{s(3),\ldots,s(k+1)\}}(2,k-3)}
  e(s;x_{1},\ldots,x_{k+1})\,\cdot\\
  \cdot e(t;x_{s(3)},\ldots,x_{s(k+1)})
   e(x_{s(1)})e(x_{s(2)})e(x_{ts(3)})\,\cdot\\
   \cdot i_{x_{s(2)}\wedge x_{s(1)}}
     i_{x_{ts(k+1)}\wedge\cdots\wedge x_{ts(5)}\wedge[x_{ts(4)},x_{ts(3)}]_S}\omega\;.
\end{multline*}
Now observe, that for any $s\in S_{k+1}$  
and shuffle $t\in Sh_{\{s(2),\ldots,s(k+1)\}}(2,k-2)$, 
the permutation $(s(1),ts(2),\ldots,ts(k+1))$ is again an element 
of $S_{k+1}$ and since there are precisely $\frac{k!}{2(k-2)!}$ many
shuffles in $Sh_{\{s(2),\ldots,s(k+1)\}}(2,k-2)$ we can just 'absorb' the first sum
over shuffles in the previous expression into the sum over general permutation. 
A similar argument holds for the second sum. After simplification this gives:
\begin{multline*}
\textstyle\frac{(k+1)k}{2}di_{x_{k+1}\wedge\cdots\wedge x_{1}}\omega\\	
-\textstyle\frac{1}{(k-1)!}\frac{k!}{2(k-2)!}
 \sum_{s\in S_{k+1}}e(s;x_{1},\ldots,x_{k+1})
  e(x_{s(1)})i_{x_{s(k+1)}\wedge\cdots\wedge x_{s(3)}\wedge[x_{s(2)},x_{s(1)}]_S}\omega\\	
+\textstyle\frac{1}{2(k-1)!}\frac{(k-1)!}{2(k-3)!}\sum_{s\in S_{k+1}}
 e(s;x_{1},\ldots,x_{k+1})e(x_{s(1)})
  i_{x_{s(k+1)}\wedge\cdots\wedge x_{s(3)}\wedge[x_{s(2)},x_{s(1)}]_S}\omega
\end{multline*}
Now we can transform this sum over arbitrary permutations back 
into a sum over shuffle permutations, such that we can
apply the definition of the Lie $(n+1)$-bracket another time.
\begin{multline*}	
\textstyle\frac{(k+1)k}{2}di_{x_{k+1}\wedge\cdots\wedge x_{1}}\omega\\	
-\textstyle\frac{k!}{(k-2)!}
 \sum_{s\in Sh(2,k-1)}e(s;x_{1},\ldots,x_{k+1})
  e(x_{s(1)})i_{x_{s(k+1)}\wedge\cdots\wedge x_{s(3)}\wedge[x_{s(2)},x_{s(1)}]_S}\omega\\	
+\textstyle\frac{(k-1)!}{2(k-3)!}
 \sum_{s\in Sh(2,k-1)}e(s;x_{1},\ldots,x_{k+1})e(x_{s(1)})
  i_{x_{s(k+1)}\wedge\cdots\wedge x_{s(3)}\wedge[x_{s(2)},x_{s(1)}]_S}\omega=
 \end{multline*}
$$
\textstyle\frac{(k+1)k}{2}di_{x_{k+1}\wedge\cdots\wedge x_{1}}\omega	
 -\frac{k!}{(k-2)!}i_{[x_1,\ldots,x_{k+1}]_{k+1}}\omega
  +\frac{(k-1)!}{2(k-3)!}i_{[x_1,\ldots,x_{k+1}]_{k+1}}\omega
$$
Again recall, that the last sum is omitted for $k=2$.
Summarizing this long computation we have shown that under the 
induction hypothesis, the equation 
\begin{multline*}
i_{[x_1,\ldots,x_{k+1}]_{k+1}}\omega=\\
\textstyle\frac{(k+1)k}{2}di_{x_{k+1}\wedge\cdots\wedge x_{1}}\omega	
 -k(k-1)i_{[x_1,\ldots,x_{k+1}]_{k+1}}\omega
  +\frac{(k-1)(k-2)}{2}i_{[x_1,\ldots,x_{k+1}]_{k+1}}\omega
\end{multline*}
is satisfied for any $k\geq 2$. This in turn leads to the equation
$$\textstyle
\left(1-\frac{(k-1)(k-2)}{2}+k(k-1)\right)i_{[x_1,\ldots,x_{k+1}]_{k+1}}\omega=
\frac{(k+1)k}{2}di_{x_{k+1}\wedge\cdots\wedge x_1}\omega	
$$
which completes the proof on homogeneous symplectic tensors and hence on arbitrary 
symplectic tensors.
\end{proof}
An immediate consequence of equation (\ref{fundamental_pairing_1}) is, 
that the set of symplectic tensors is closed under the operation of all 
higher tensor brackets: 
\begin{corollary}$S\overline{ym}(A,\mathfrak{g},\omega)$ is a sub Lie $\I$-algebra of 
$\left(E(A,\mathfrak{g}),[\cdot,\ldots,\cdot]_{k\in\N}\right)$.
\end{corollary}
\begin{proof}Since $L_x\omega=0$ is the defining equation of a symplectic tensor and  
Lie derivation is $\R$-linear in both arguments,
$S\overline{ym}(A,\mathfrak{g},\omega)$ is a graded vector subspace of $E(A,\mathfrak{g})$.
To see that it is moreover closed under the operation of 
all brackets, we use (\ref{fundamental_pairing_1}) to compute
$L_{[x_1,\ldots,x_k]_k}\omega=di_{[x_1,\ldots,x_k]_k}\omega =0$ for any
$x_1,\ldots,x_k\in Sym(A,\mathfrak{g},\omega)$.
\end{proof}
\begin{remark}$S\overline{ym}(A,\mathfrak{g},\omega)$ is in general
not a sub $A$-module of $E(A,\mathfrak{g})$. 
In fact since $L_{a\cdot x}\omega=da\smwedge i_x\omega$, 
it is a sub $A$-module, if and only if $da=0$ for any $a\in A$. This is 
well known from symplectic geometry.
\end{remark}
\begin{corollary}The set of Hamiltonian tensors 
is an ideal in the Lie $\I$-algebra $S\overline{ym}(A,\mathfrak{g},\omega)$.  
\end{corollary}
\begin{proof}
This is an immediate consequence of equation (\ref{fundamental_pairing_1}), because
the higher tensor brackets of all symplectic tensors are Hamiltonian.
\end{proof}
Since we consider symplectic tensors of arbitrary degree and potentially
very high degenerated n-plectic cocycles, the kernel of
$\omega$ can be very large. However it is always an ideal in Hamiltonian 
and hence in symplectic tensors, as the following corollary shows: 
\begin{corollary}The kernel $\ker(\omega)$ of the $n$-plectic cocycle
is an ideal in the Lie $\I$-algebra of Hamiltonian tensors.
\end{corollary}
\begin{proof}We have to show $[x_1,\ldots,x_k]_k\in \ker(\omega)$
for all $k\in\N$, if at least one argument $x_i$ is element of the kernel 
of $\omega$, but this follows from (\ref{fundamental_pairing_1}), since
$i_{[x_1,\ldots,x_k]_k}\omega=di_{x_k\wedge\cdots\wedge x_1}\omega=0$, 
for some $x_i\in\ker(\omega)$.
\end{proof}
In what follows we will consider symplectic tensors only up to elements of the kernel 
of $\omega$. Since the kernel is an ideal, we can divide it out and
write
\begin{equation}
Sym(A,\mathfrak{g},\omega):=S\overline{ym}(A,\mathfrak{g},\omega)/\ker(\omega)
\end{equation}
for the Lie $\I$-algebra of symplectic tensors modulo elements of the
kernel of $\omega$. According to a better readable text we just write $x$ instead of 
$[x]$ for an appropriate equivalence class. Contraction and Lie derivative of 
$\omega$ along such a tensor, do not depend on a particular representative.  
\subsection{N-plectic extensions of symplectic tensors}
We show that any choice of an $n$-plectic cocycle 
gives rise to an extension of the Lie $\I$-algebra of symplectic 
tensors by the Chevalley-Eilenberg complex of cotensors.

Usually this cochain complex of cotensors is not seen 
as a Lie $\I$-algebra, but as a differential graded algebra. 
However with respect to the tensor grading we can see it 
as an \textit{abelian} Lie $\I$-algebra, where all brackets except the unary are zero.

To define this extension we need the \textit{shift} $A_{[n]}$ of a
(co)chain complex $A$, which is again a (co)chain complex given by
$A_{[n]}^k = A^{k + n}$ and 
$d_{A_{[n]}}^k := (-1)^{n} d_A^{k+n}$.  
\begin{definition}[The n-plectic extension]\label{n-plectic extension}
Let $(A,\mathfrak{g},\omega)$ be an $n$-plectic structure,
$Sym(A,\mathfrak{g},\omega)$ the Lie $\I$-algebra of symplectic 
tensors up to $\ker(\omega)$ and $E(A,\mathfrak{g}^\vee)_{[n]}$ the 
abelian Lie $\I$-algebra of cotensors, but with the tensor degree shifted by $n$. 
Then the map
\begin{equation}
d_\omega : 
 E(A,\mathfrak{g}^\vee)_{[n]}\oplus Sym(A,\mathfrak{g},\omega) \to  
  E(A,\mathfrak{g}^\vee)_{[n]}\oplus Sym(A,\mathfrak{g},\omega)
\end{equation}
defined by 
$d_\omega\left(f_{[n]},x\right)=
 \left(\left(i_x\omega-df\right){}_{[n]},0\right)$
is called the $\omega$\textbf{-extension of the exterior derivative}. 

Moreover, let $B_{j\in\NN}$ be the sequence of Bell numbers and $k\in\N$ with
$k\geq 2$. Then the map
\begin{equation}
\{\cdot,\ldots,\cdot\}^k_\omega: 
 \textstyle\bigtimes^k \left(E(A,\mathfrak{g}^\vee)_{[n]}
  \oplus Sym(A,\mathfrak{g},\omega)\right)
   \to E(A,\mathfrak{g}^\vee)_{[n]}\oplus Sym(A,\mathfrak{g},\omega)
\end{equation}
defined by
\begin{equation}
\{(f^1_{[n]},x_1),\ldots,(f^k_{[n]},x_k)\}^k_\omega=  
\left({B_{k-1}}\cdot i_{x_k\wedge\cdots \wedge x_1}\omega_{[n]},
 [x_1,\ldots,x_k]_k\right)
\end{equation}
is called the $\omega$\textbf{-extension of the $k$-th tensor bracket}.
\end{definition}
Recall from definition \ref{tensor_grading} that we consider cotensors as graded but
concentrated in non positive degrees. This implies that the
$n$-shifted cotensor $f_{[n]}$ of any homogeneous 
$f\in E(A,\mathfrak{g}^\vee)$ has tensor degree $(n-|f|)$.
  
The following theorem shows, that this gives indeed a 
Lie $\I$-algebra, which can be seen as an extension of the tensor 
Lie $\I$-algebra by the Chevalley-Eilenberg cochain complex along 
the n-plectic cocycle: 
\begin{theorem}\label{extended_brackets}
The graded direct sum 
$E(A,\mathfrak{g}^\vee)_{[n]}\oplus Sym(A,\mathfrak{g},\omega)$ 
of graded vector spaces is a Lie $\I$-algebra with respect to the
n-plectic extensions of the exterior derivative and the higher tensor brackets.
\end{theorem}
\begin{proof}Since the contraction and all tensor brackets are graded linear, 
so are their extensions.
To see that they are homogeneous of tensor degree $-1$, 
suppose that $(f^1_{[n]},x_1),\ldots,(f^k_{[n]},x_k)\in 
E(A,\mathfrak{g}^\vee)_{[n]}\oplus Sym(A,\mathfrak{g},\omega)$ are 
homogeneous. Then $|x_i|=|f^i_{[n]}|=n-|f^i|$ and from the homogeneity of 
the higher tensor bracket follows
\begin{align*}
|x_1|+\ldots+|x_k|-1
&=|[x_1,\ldots,x_k]_k|\\
&=n + (-(n+1)+(|x_1|+\ldots+|x_k|))\\
&=n + (|\omega|+(|x_1|+\ldots+|x_k|))\\
&=n+|i_{x_k\wedge\cdots\wedge x_1}\omega|\\
&=|(i_{x_k\wedge\cdots\wedge x_1}\omega)_{[n]}|\;.
\end{align*}
Therefore the $k$-ary extended bracket is homogeneous of tensor degree $-1$. Similar
we compute
$|(i_x\omega-df)_{[n]}|=n+|i_x\omega-df|=n+|i_x\omega|=n+(|x|-(n+1))=|x|-1$ for
$|x|=f_{[n]}$ which means, that $d_\omega$ is homogeneous of degree $-1$, too. 

Graded symmetry follows from the graded symmetry of the higher tensor brackets and
the exterior product.

To proof the first Jacobi equation, we have to show
$d_\omega^2=0$. Therefore recall that $di_x\omega=0$ on symplectic tensors
and compute
\begin{align*}
d_\omega\left(f_{[n]},x)\right)^2
&=d_\omega\left((i_x\omega-df){}_{[n]},0\right)\\
&=\left((-di_x\omega+d^2f)_{[n]},0\right)\\
&=(0,0)\;.
\end{align*}
To calculate the other Jacobi equations, observe
$\{d_\omega(f^1_{[n]},x_1),\ldots,(f^k_{[n]},x_k)\}^k_\omega=(0,0)$, for any
$k\geq 2$, since the extended bracket vanish if at least one argument 
contains the zero symplectic tensor up to elements of $\ker(\omega)$.

Using this we can simplify the second Jacobi equation according to
\begin{multline*}
d_\omega(\{(f^1_{[n]},x_1),(f^2_{[n]},x_2)\}^2_\omega)
 +\{d_{\omega}(f_{[n]}^{1},x_{1}),(f_{[n]}^{2},x_{2})\}^2_\omega\\
  +e(x_1,x_2)\{d_{\omega}(f_{[n]}^{2},x_{2}),(f_{[n]}^{1},x_{1})\}^2_\omega=
   d_\omega(\{(f^1_{[n]},x_1),(f^2_{[n]},x_2)\}^2_\omega)\;.
\end{multline*}
To see that the single term on the right vanishes too, we use equation 
(\ref{fundamental_pairing_1}) and $B_1=1$ to compute
\begin{align*}
d_\omega(\{(f_{[n]}^{1},x_{1}),(f_{[n]}^{2},x_{2})\}^2_\omega)
&=d_\omega\left((B_1\,i_{x_{2}\wedge x_{1}}\omega)_{[n]},[x_{1},x_{2}]_2\right)\\
&=\left((i_{[x_{1},x_{2}]_2}\omega-B_1\,di_{x_{2}\wedge x_{1}}\omega)_{[n]},0\right)\\
&=(0,0)\;.
\end{align*}
It remains to show that the general weak Jacobi equation vanishes for any $k\geq 3$, that is 
\begin{multline*}
\textstyle\sum_{p+q=k+1}\sum_{s\in Sh(p,q-1)}e(s;x_1,\ldots,x_k)\,\cdot\\
\cdot\left\{\left\{(f^1_{[n]},x_1),\ldots,(f^p_{[n]},x_p)\right\}_p,
 (f^{p+1}_{[n]},x_{p+1}),\ldots,(f^k_{[n]},x_k)\right\}_q=(0,0)
\end{multline*} 
is satisfied. To see that, recall that all terms for $p=1$ vanishes, since
the extended brackets are zero, if at least one argument contains the zero
symplectic tensor. In addition the term for $q=1$ becomes 
$$
((i_{[x_1,\ldots,x_k]_k}\omega 
-B_{k-1}di_{x_k\wedge\cdots\wedge x_1}\omega)_{[n]},0)\;.
$$
Using the definition of the n-plectic extended brackets, 
the remaining sum of the weak Jacobi expression can be written as
\begin{multline*}
\textstyle\sum_{p+q=k+1}^{p,q\geq 2}\sum_{s\in Sh(p,q-1)}e(s;x_1,\ldots,x_k)\;\cdot\\
 \cdot\left((B_{q-1}\,i_{x_{s(k)}\wedge\cdots\wedge x_{s(p+1)}\wedge 
  [x_{s(1)},\ldots,x_{s(p)}]_p}\omega)_{[n]},
  [[x_1,\ldots,x_p]_p,x_{p+1},\ldots,x_k]_q\right)
\end{multline*}
and taken the weak Jacobi identity of the higher tensor brackets
into account, this can be expressed as
\begin{multline*}
\textstyle\sum_{p+q=k+1}^{p,q\geq 2}\sum_{s\in Sh(p,q-1)}e(s;x_1,\ldots,x_k)\cdot\\
 \cdot\left((B_{q-1}\,i_{x_{s(k)}\wedge\cdots\wedge x_{s(p+1)}\wedge 
  [x_{s(1)},\ldots,x_{s(p)}]_p}\omega)_{[n]},0\right)
\end{multline*}
It follows, that the Jacobi identity holds in dimension $k$, if and 
only if the equation
\begin{multline*}
B_{k-1}di_{x_k\wedge\cdots\wedge x_1}\omega -i_{[x_1,\ldots,x_k]_k}\omega =\\
\textstyle\sum_{p+q=k+1}^{p,q\geq 2}\sum_{s\in Sh(p,q-1)}e(s;x_1,\ldots,x_k)
 B_{q-1}\,i_{x_{s(k)}\wedge\cdots\wedge x_{s(p+1)}\wedge 
  [x_{s(1)},\ldots,x_{s(p)}]_p}\omega
\end{multline*}
is satisfied.

To simplify this equation we use (\ref{fundamental_pairing_1}), and the graded symmetry 
of the exterior product and the higher tensor brackets, to rewrite it into
\begin{multline*}
B_{k-1}di_{x_k\wedge\cdots\wedge x_1}\omega -di_{x_k\wedge\cdots\wedge x_1}\omega =\\
\textstyle\sum_{p+q=k+1}^{p,q\geq 2}\frac{1}{p!(q-1)!}B_{q-1}\sum_{s\in S_k}
 e(s;x_1,\ldots,x_k)i_{x_{s(k)}\wedge\cdots\wedge x_{s(p+1)}\wedge 
  [x_{s(1)},\ldots,x_{s(p)}]_p}\omega\;.
\end{multline*}
Then we substitute the tensor $p$-bracket by its definition. 
(According to a better readable text we write $Sh_S(p,q)$ for 
the set of $(p,q)$-shuffle permutations defined explicit on the finite
set $S$.) This transforms the equation into
\begin{multline*}
\left(B_{k-1}-1\right)di_{x_k\wedge\cdots\wedge x_1}\omega=\\
\textstyle
\sum_{p+q=k+1}^{p,q\geq 2}\frac{1}{p!(q-1)!}B_{q-1}\sum_{s\in S_{k}}
 \sum_{t\in Sh_{\{s(1),\ldots,s(p)\}}(2,p-2)}e(s;x_{1},\ldots,x_{k})\,\cdot\\
\cdot e(t;x_{s(1)},\ldots,x_{s(p)})
 e(x_{ts(1)})i_{x_{s(k)}\wedge\cdots\wedge x_{s(p+1)}\wedge 
  x_{ts(p)}\wedge\cdots\wedge x_{ts(3)}\wedge[x_{ts(2)},x_{ts(1)}]_S}\omega
\end{multline*} 
Now observe, that for any $s\in S_{k}$  
and shuffle $t\in Sh_{\{s(1),\ldots,s(p)\}}(2,p-2)$, 
the permutation $(ts(1),\ldots,ts(p),s(p+1),\ldots,s(k))$ is again an element 
of $S_{k}$. Since there are precisely $\frac{p!}{2(p-2)!}$ many
shuffles in $Sh_{\{s(1),\ldots,s(p)\}}(2,p-2)$ we can just 'absorb' the sum
over these shuffles in the previous equation into the sum over general permutation:
\begin{multline*}
\left(B_{k-1}-B_0\right)di_{x_k\wedge\cdots\wedge x_1}\omega=\\
\textstyle
\sum_{p+q=k+1}^{p,q\geq 2}\frac{B_{q-1}}{2(q-1)!(p-2)!}\sum_{s\in S_{k}}
 e(s;x_{1},\ldots,x_{k})
 e(x_{s(1)})i_{x_{s(k)}\wedge\cdots\wedge x_{s(3)}\wedge[x_{s(2)},x_{s(1)}]_S}\omega.
\end{multline*} 
Then we transform the summation over arbitrary permutations back into 
a sum over $(2,k-2)$-shuffles, such that we can apply the definition of 
the tensor $k$-bracket again. This gives
\begin{multline*}
\left(B_{k-1}-B_0\right)di_{x_k\wedge\cdots\wedge x_1}\omega=
 \textstyle \sum_{p+q=k+1}^{p,q\geq 2}\frac{(k-2)!}{(q-1)!(p-2)!}B_{q-1}\,\cdot\\
  \cdot\textstyle\sum_{s\in Sh(2,(k-2))}e(s;x_{1},\ldots,x_{k})e(x_{s(1)})
 i_{x_{s(k)}\wedge\cdots\wedge x_{s(3)}\wedge[x_{s(2)},x_{s(1)}]_S}\omega
\end{multline*}
and after using the definition of the $k$-ary tensor bracket as well
as (\ref{fundamental_pairing_1}), this rewrites into 
\begin{align*}
\left(B_{k-1}-B_0\right)di_{x_k\wedge\cdots\wedge x_1}\omega=
&\textstyle\sum_{p+q=k+1}^{p,q\geq 2}\frac{(k-2)!}{(q-1)!(p-2)!}B_{q-1}
 i_{[x_1,\ldots,x_k]_k}\omega=\\
&\textstyle\sum_{p+q=k+1}^{p,q\geq 2}\frac{(k-2)!}{(q-1)!(p-2)!}B_{q-1}
 di_{x_k\wedge\cdots\wedge x_1}\omega\;.
\end{align*}
Summarizing this computation, the weak Jacobi equation in dimension $k$ is satisfied, if 
and only if 
$$
\textstyle\sum_{q=2}^{k-1}\frac{(k-2)!}{(q-1)!(k-1-q)!}B_{q-1}
 =B_{k-1}-B_0
$$
but this equation holds for all $k\geq 3$, as we can see from 
the recurrence relation 
$B_{k+1}=\sum_{p=0}^{k} \binom{k}{p} B_p$ of the Bell numbers.
Therefore all weak Jacobi equations are satisfied and the n-plectic extension
is a Lie $\I$-algebra.
\end{proof}
\begin{corollary} 
Considering the Chevalley-Eilenberg complex of exterior cotensors as a 
Lie $\I$-algebra with only non vanishing unary bracket, the diagram
\begin{equation}
\xymatrix{
E(A,\mathfrak{g}^\vee)_{[n]} \ar@{^{(}->}[r] & 
   E(A,\mathfrak{g}^\vee)_{[n]}\oplus Sym(A,\mathfrak{g},\omega) \ar@{->>}[r] & 
   Sym(A,\mathfrak{g},\omega)
}
\end{equation}
is a short exact sequence of Lie $\I$-algebras, where the morphisms
are the natural inclusion and projection, respectively.
\end{corollary}
\subsection{The Hamiltonian Cohomology}
On a connected symplectic manifold any two
Hamiltonian functions associated to the same
Hamiltonian vector field have equal 
exterior derivatives and hence differ by a constant only.

From another perspective, we could say that 
Hamiltonian functions are only defined up to locally constant functions, 
or that Hamiltonian functions associated to the same vector 
field differ by closed 0-forms only.

We generalize this idea to Hamiltonian tensors of 
possible higher tensor degrees, by passing to the cohomology of 
$E(A,\mathfrak{g}^\vee)_{[n]}\oplus Sym(A,\mathfrak{g},\omega)$
with respect to the coboundary map $d_\omega$.
The cocycle property of $d_\omega$ then provides the fundamental pairing 
between Hamiltonian tensors and cotensors and passing to the cohomology
takes care of the previously mentioned ambiguity inherent in this pairing.

The central part of this work is to show that there is a 
Lie $\I$-structure on this cohomology, wich generalizes
the usual \textit{Poisson Lie bracket} of symplectic functions to the general
higher context.

To start we first look at the cocycles and coboundaries of the n-plectic 
extended derivation $d_\omega$. As the following proposition shows the
cocycle property is then nothing but the fundamental pairing well known 
from symplectic geometry:
\begin{proposition}Let $(A,\mathfrak{g},\omega)$ be an $n$-plectic structure 
with n-plectic extension
$E(A,\mathfrak{g}^\vee)_{[n]}\oplus Sym(A,\mathfrak{g},\omega)$. Then a pair
$(f_{[n]},x)\in E(A,\mathfrak{g}^\vee)_{[n]}\oplus Sym(A,\mathfrak{g},\omega)$
is a cocycle with respect to $d_\omega$, if and only if
\begin{equation}\label{funda_pairing}
i_x\omega=df\;.
\end{equation}
It is moreover a coboundary, if and only if
$x\in\ker(\omega)$ and there is a symplectic tensor 
$y\in Sym(A,\mathfrak{g},\omega)$
as well as a cotensor $h\in E(A,\mathfrak{g}^\vee)$, such that $f=i_y\omega - dh$. Two 
cocycles $(f^1_{[n]},x_1)$, $(f^2_{[n]},x_2)$ 
are cohomologous, precisely if $x_1-x_2\in\ker(\omega)$ and 
$f_1-f_2=i_y\omega +dh$.
\end{proposition}
\begin{proof}Apply the definition of $d_\omega$ and recall that 
$0\in Sym(A,\mathfrak{g},\omega)$ is the kernel of $\omega$.
\end{proof}
Now passing to the cohomology, restricts 
the n-plectic extension of symplectic tensors precisely to the
Hamiltonian tensors, together with particular cotensors, connected
by the fundamental pairing, up to closed forms. 
\begin{definition}[Hamiltonian Cohomology]
Let $(A,\mathfrak{g},\omega)$ be an $n$-symplectic structure,
$E(A,\mathfrak{g}^\vee)_{[n]}\oplus Sym(A,\mathfrak{g},\omega)$ the 
n-plectic extension of symplectic tensors and 
$d_\omega$ the n-plectic extension of the exterior derivative.
Then the graded $\R$-vector space 
$$
H(A,\mathfrak{g},\omega):= \boplus_{k\in\Z} 
 \ker(d^k_\omega) / \mathrm{im}(d^{k+1}_\omega)
$$
is called the \textbf{Hamiltonian cohomology} of the $n$-plectic
structure. Cohomology classes 
$[f_{[n]},x]\in H(A,\mathfrak{g},\omega)$ are called 
(pairs of) \textbf{Hamiltonian tensors and cotensors} and the defining
equation of a cocycle
\begin{equation}
i_x\omega= df
\end{equation}
is called the \textbf{fundamental pairing} of Hamiltonian tensors and
cotensors.
\end{definition} 
Elements of a cohomology class 
$[f_{[n]},x]\in H^k(A,\mathfrak{g},\omega)$ are pairs of Hamiltonian tensors,
homogeneous of tensor degree $k$ and cotensors, homogeneous of
tensor degree $(n-k)$, linked by the equation
$i_x\omega = df$. All representative tensors are equal up to elements of the kernel of 
$\omega$ and all such cotensors are equal up to (certain) closed forms.

The next proposition shows, the Hamiltonian cohomology complex is bounded:
\begin{proposition}Let $(A,\mathfrak{g},\omega)$ be an $n$-plectic structure. Then
the Hamiltonian cohomology is an $\NN$-graded complex, where in particular
the $k$-th cohomology $H^k(A,\mathfrak{g},\omega)$ is trivial
for all $k>(n+1)$.
\end{proposition}
\begin{proof}To see $H^{-k}(A,\mathfrak{g},\omega)=\{0\}$ for all
$k\in\N$, observe that any representative Hamiltonian tensor $x$ of a 
cohomology class $[f_{[n]},x]\in H^{-k}(A,\mathfrak{g},\omega)$ has tensor
degree $-k$ and is therefore zero. From the cocycle condition 
$i_x\omega = df$ then follows that any representative Hamiltonian
cotensor $f$ has to satisfy $df=0$, which means, that it is a coboundary
with respect to $d_\omega$. Hence the class is the zero class.

To see the other bound, suppose $[f_{[n]},x]\in H^k(A,\mathfrak{g},\omega)$ 
for some $k>(n+1)$. 
Then the tensor degree of a any representative tensor $x$ is
$|x|=k>(n+1)$ and hence $x\in\ker(\omega)$. From the cocycle condition 
$i_x\omega=df$ then follows $df=0$ and again this means, that $f$ is a coboundary
with respect to $d_\omega$ and that the class is the zero class. 
\end{proof}
As the central part of this work, we now generalize the usual Poisson 
Lie bracket of Hamiltonian functions to a sequence of 'higher brackets' on the
Hamiltonian cohomology of any n-plectic structure and show that this
arranges into a Lie $\I$-algebra.
\begin{definition}[Poisson brackets]
Let $(A,\mathfrak{g},\omega)$ be an $n$-plectic
structure, with Hamiltonian cohomology $H(A,\mathfrak{g},\omega)$,
$[\cdot,\ldots,\cdot]_{k\in\N}$ the sequence of higher tensor brackets
(\ref{Higher_Lie_bracket}). Then the map
\begin{equation}
\{\cdot,\ldots,\cdot\}_k: H(A,\mathfrak{g},\omega)\times\cdots\times 
 H(A,\mathfrak{g},\omega) \to  H(A,\mathfrak{g},\omega)
\end{equation}
defined for any $k\in\N$ and cohomology classes 
$[f^1_{[n]},x_1],\ldots,[f^k_{[n]},x_k]\in H(A,\mathfrak{g},\omega)$
by the equation
\begin{equation}
\{[f^1_{[n]},x_1],\ldots,[f^k_{[n]},x_k]\}_k =
  [\;i_{x_k\wedge\cdots\wedge x_1}\omega_{[n]}\,,\,[x_1,\ldots,x_k]_k\;]\,,
\end{equation}
is called the \textbf{Poisson Lie k-bracket} (or just $k$-ary Poisson Lie bracket) 
of the Hamiltonian cohomology.
\end{definition}
As the following theorem show, these brackets are well defined and 
arrange the Hamiltonian cohomology into a Lie $\I$-algebra.
\begin{theorem}The Hamiltonian cohomology $H(A,\mathfrak{g},\omega)$
is a Lie $\I$-algebra, with respect to the sequence 
$\{\cdot,\ldots,\cdot\}_{k\in\N}$ of Poisson Lie k-brackets.
\end{theorem}
\begin{proof}Since the Hamiltonian tensors $x$ of all representatives
$(f_{[n]},x)$ of a cohomology class 
$[f_{[n]},x]\in H(A,\mathfrak{g},\omega)$ are equal up to elements
of the kernel of $\omega$, 
the bracket does not depend on the particular chosen representative. 

To see that the unary bracket $\{\,\cdot\,\}_1$ is the zero operator, recall from, 
that the tensor bracket $[\,\cdot\,]_1$ is zero and that
$i_x\omega$ is closed. Then
$$\{[f_{[n]},x]\}_1=[i_x\omega,[x]_1]=[i_x\omega,0]=[0,0]\,.$$
Since $\{\,\cdot\,\}_1$ is zero, it only remains to show that the weak 
Jacobi equations
\begin{multline*}
\textstyle\sum_{p+q=k+1}^{p,q\geq 2}\sum_{s\in Sh(p,q-1)}e(s;x_1,\ldots,x_k)\,\cdot\\
\cdot\left\{\left\{[f^1_{[n]},x_1],\ldots,[f^p_{[n]},x_p]\right\}_p,
[f^{p+1}_{[n]},x_{p+1}],\ldots,[f^k_{[n]},x_k]\right\}_q=[\,0\,,0\,]
\end{multline*}
are satisfied for all $k\geq 3$ as well as 
$[f^1_{[n]},x_1],\ldots,[f^k_{[n]},x_k]\in H(A,\mathfrak{g},\omega)$ and arguing
similar to the proof of theorem \ref{extended_brackets} we need to show that 
\begin{multline*}
\textstyle\sum_{p+q=k+1}^{p,q\geq 2}\sum_{s\in Sh(p,q-1)}e(s;x_1,\ldots,x_k)
 (i_{x_{s(k)}\wedge\cdots\wedge x_{s(p+1)}\wedge 
  [x_{s(1)},\ldots,x_{s(p)}]_p}\omega)_{[n]}
\end{multline*}
is closed, since $[0,0]$ is the equivalence class of elements of the
kernel of $\omega$ and closed forms. Again we know from the proof of 
theorem \ref{extended_brackets} that this expression can be rewritten into
$$
\textstyle\sum_{p+q=k+1}^{p,q\geq 2}\frac{(k-2)!}{(q-1)!(p-2)!}
 di_{x_k\wedge\cdots\wedge x_1}\omega
$$
which then completes the proof. 
\end{proof}
\begin{remark}
Since the Poisson Lie $1$-bracket $\{\cdot\}_1$ is zero on the Hamiltonian
cohomology, the general Jacobi equation (\ref{sh_Jacobi}) simplifies for $n=3$
into the usual (strict) Jacobi equation of a graded Lie algebra. This is an
important fact as it allows Lie algebra representations and in particular momentum 
maps to be handled similar to the common symplectic case.

However the higher brackets don't vanish and so more general (weak) 
representations can be expected which are not yet part of symplectic
geometry. 
\end{remark}
\subsection{Momentum maps}In order to study Lie group actions on 
symplectic manifolds, the momentum map of
such an action was originally defined as a certain map from the
manifold into the dual of the appropriate Lie algebra. 
However it became clear that momentum maps can equivalently be seen as certain
morphisms $J:\mathfrak{g}\to C^\I(M)$ of Lie algebras,
which fit into otherwise exact, commutative diagrams 
$$
\xymatrix{
 & & & \mathfrak{g}\ar[d] \ar@{-->}[dl]_J\\ 
 0 \ar[r] & H^0(M) \ar[r] &  C^\I(M) \ar[r] &  Sym(M)\ar[r] &
  H^1(M) \ar[r] & 0
}
$$
in the category of Lie algebras, where $Sym(M)$ is the Lie algebra
of symplectic vector fields, $C^\I(M)$ a Lie algebra with
respect to the symplectic Poisson Lie bracket and the de-Rham
cohomologies $H^i(M)$ are abelian Lie algebras.

In \cite{FRZ} the authors generalizes this to an appropriate diagram in
Lie $\I$-algebras, where the function Poisson-algebra is replaced by a 
certain Lie $\I$-algebra, that can be seen as a sub structure of our
n-plectic tensor extension. We continue their generalization
to symplectic tensors and cotensors in a wider range of tensor degrees.

In the symplectic setting a momentum map (if it exists) is in general 
\textbf{only a morphism up to} $H^0(M)$. Translated to
higher symplectic geometry, this should be a morphism 
into the Hamiltonian cohomology, which then is a 
\textbf{morphism up to closed forms} only. 

Moreover since the Hamiltonian cohomology is internal to Lie $\I$-algebras, this
kind of momentum map makes sense for more general Lie $\I$-algebras, not just
for Lie algebras associated to a particular Lie group action. 

Putting this together, we are able to give a conceptual very simple definition
of a momentum map:
\begin{definition}Let $(A,\mathfrak{g},\omega)$ be an $n$-plectic
structure, with Hamiltonian cohomology $H(A,\mathfrak{g},\omega)$ and 
$(L,D_{k\in\N})$ a Lie $\I$-algebra.
Then an \textbf{n-plectic momentum map} is a weak morphism of 
Lie $\I$-algebras 
\begin{equation}
J_{\I}: L \to H(A,\mathfrak{g},\omega)\;.
\end{equation}
\end{definition}
The structure equations for general weak Lie $\I$-morphisms (see Appendix) 
take care of the properties one would expect from a momentum map. 
The interested reader is encouraged 
to show that this resamples precisely the common definition of a momentum map, 
in case $L$ is the Lie algebra of a symplectic action on an appropriate
manifold.
\section{Conclusion and Outlook}
We developed a general definition of higher symplectic structures and defined a
reasonable momentum map, to study Lie $\I$-algebra representations in this context.
However a well understood obstruction theory for momentum maps as well as a good
notion of 'higher symplectic morphisms' has still to be found. 
In \cite{FRZ} the authors started to look at the first of these problems. 
\begin{appendix}
\section{Lie $\I$-algebras} We recall the most basic stuff
about Lie $\I$-algebras. There are many incarnations of them 
\cite{LV}, but we will only look at their 
graded symmetric, 'many brackets' version, since that picture fits nicely into
the Schouten calculus and is moreover useful when it comes to actual computations. 

Lie $\I$-algebras are defined on $\Z$-graded vector spaces and 
consequently we recall them first:
\subsection{Graded Vector Spaces}
In what follows $\mathbb{K}$ will always be a field and $\Z$ the Abelian
group of integers with respect to addition. 
A $\Z$-graded $\mathbb{K}$-vector space $V$ is the direct sum 
$\oplus_{n\in \Z} V_n$ of $\mathbb{K}$-vector spaces $V_n$.  
Since this is a coprodut, there are natural injections 
$i_n:V_n \to V$ and a vector is called \textbf{homogeneous of degree} $n$ 
if it is in the image of the injection $i_n$. 
In that case we write $deg(v)$ or $|v|$ for its degree.

According to a better readable text we just write graded vector space 
as a shortcut for $\Z$-graded $\mathbb{K}$-vector space.

A morphism $f : V \to W$ of graded vector spaces, homogeneous of degree $r$,  
is a sequence of linear maps $f_n : V_n \to W_{n+r}$ for any $n\in \Z$ and
the integer $r\in\Z$ is called the degree of $f$, denoted by $deg(f)$ (or $|f|$). 

For any $n\in\N$, an $n$-multilinear map $f: V_1 \times \cdots \times V_n \to W$,
homogeneous of degree $r$ is a sequence of $n$-multilinear maps 
$f_{k} : (V_1)_{n_1} \times \ldots \times (V_k)_{n_k} \to W_{\sum n_i+r}$ for all $j_i\in \Z$ with $\sum j_i=k$.

The $\Z$-graded tensor product $V \otimes W$ of two graded vector spaces 
$V$ and $W$ is given by
$$
\textstyle\left(V \otimes W \right)_n :=
\oplus_{i+j=n}\left( V_i \otimes W_j\right)
$$
and the Koszul commutativity constraint
$\tau: V \otimes W \to  W \otimes V$ is on homogeneous elements 
$v\otimes w \in V \otimes W$ defined 
by 
$$\tau(v \otimes w):=(-1)^{deg(v)deg(w)} w \otimes v$$ and then extended to
$V\otimes W$ by linearity.

\begin{remark}
We define the symbols $e(v):=(-1)^{deg(v)}$, 
$e(v,w):=(-1)^{deg(v)deg(w)}$. The \textbf{Koszul sign} 
$e(s;v_1,\ldots,v_k) \in \{-1,+1\}$ is defined for any permutation $s\in S_k$ and
any homogeneous vectors $v_1,\ldots,v_k\in V$ by 
\begin{equation}\label{Koszul_convention}
v_1\otimes \ldots \otimes v_k= e(s;v_1,\ldots,v_k) 
 v_{s(1)}\otimes \ldots \otimes v_{s(k)}.
\end{equation}
In an actual computation it can be determined by the following rules: 
When a permutation $s\in S_k$ is a transposition  $j\leftrightarrow j+1$
of consecutive neighbors, 
then $e(s;v_1,\ldots,v_k)= (-1)^{deg(v_j)\cdot deg(v_{+1})}$ and 
if $t\in S_k$ is another permutation, then
$e(ts;v_1,\ldots,v_k)=e(t;v_{s(1)},\ldots,v_{s(k)})e(s;v_1,\ldots,v_k)$.
\end{remark}
A graded $k$-linear morphism $ f: \bigtimes^k V \to W$ is called 
\textbf{graded symmetric} if 
$$f(v_1,\ldots,v_k) = e(s;v_1,\ldots,v_k)f(v_{s(1)},\ldots,v_{s(k)})$$
for all $s\in S_k$. 
\subsection{Shuffle Permutation}
Let $S_k$ be the symmetric group, i.e the group of all bijective maps 
of the ordinal $\ordinal{k}$.
\begin{definition}[Shuffle Permutation] For any $p,q\in \N$
a $(p,q)$-shuffle is a permutation 
$s\in S_{p+q}$ with $s(1)<\ldots<s(p)$ and $s(p+1)<\ldots<s(p+q)$.
We write $Sh(p,q)$ for the set of all $(p,q)$-shuffles.

More generally for any $p_1,\ldots,p_n\in\N$ a $(p_1,\ldots,p_n)$-shuffle
is a permutation $s\in S_{p_1+\cdots+p_n}$ with  
$s(p_{j-1}+1)<\ldots<s(p_{j-1}+p_{j})$. 
We write $Sh(p_1,\ldots,p_n)$ for the set of all $(p_1,\ldots,p_n)$-shuffles.
\end{definition}
\subsection{Lie $\I$-algebas} On the structure level Lie $\I$-algebras generalize 
(differential graded) Lie-algebras to a setting where the Jacobi identity isn't 
satisfied any more, but holds up to particular higher brackets. This can be defined 
in many different ways \cite{LV}, 
but the one that works best for us is its 'graded symmetric, many bracket' version.
\begin{definition} A \textbf{Lie $\I$-algebra} $\left(V,(D_k)_{k\in \N}\right)$
is a $\Z$-graded $\R$-vector space $V$, together with a sequence $(D_k)_{k\in \N}$ of
graded symmetric, $k$-multilinear maps $D_k : \bigtimes^k V \to V$, 
homogeneous of of degree $-1$, such that the \textbf{weak Jacobi equations}
$$\label{sh_Jacobi}
\textstyle\sum_{i+j=n+1} \left(\sum_{s\in Sh(j,n-j)}
e\left(s;v_1,\ldots,v_n\right)D_i\left(D_j \left(
v_{s_1}, \ldots, v_{s_j} \right), v_{s_{j+1}}, \ldots, v_{s_n}\right)\right) = 0
$$
are satisfied for any integer $n \in \N$ and any vectors $v_1,\ldots,v_n \in V$. 
\end{definition} 
In particular Lie $\I$-algebras generalizes ordinary Lie algebras, if the grading is chosen 
right:
\begin{example}[Lie Algebra] Every Lie algebra $\left(V,[\cdot,\cdot]\right)$ is a  
Lie $\I$-algebra if we consider $V$ as concentrated in degree one and define $D_k=0$ for any $k \neq 2$ as well as $D_2(\cdot,\cdot):=[\cdot,\cdot]$.
\end{example}
Very different from common Lie theory is, that a morphism
of Lie $\I$-algebras is not necessarily just a single map. In fact such a 
morphism is a \textit{sequence} of maps, satisfying a particular structure
equation.
\begin{definition}For any two Lie $\I$-algebras $(V,(D_k)_{k\in\N})$ and 
$(W,(l_k)_{k\in\N})$ a \textbf{morphism of Lie $\I$-algebras} is a sequence
$(f_k)_{k\in\N}$ of graded symmetric, $k$-multilinear maps
$$
f_k: V \times \cdots \times V \to W
$$
homogeneous of degree $0$, such that the structure equations
\begin{multline*}\label{lie-infty-morph}
\textstyle\sum_{p+q=n+1}\left({\sum_{s\in Sh(q,p-1)}e\left(s\right)}
f_p\left(D_{q}\left(v_{s(1)},\ldots,v_{s(q)}\right),
v_{s(q+1)},\ldots, v_{s(n)}\right)\right)=\\
\textstyle\sum_p\frac{1}{p!}\sum^{k_1+\ldots +k_p =n}_{s\in Sh(k_1,\ldots,k_p)}
e\left(s\right)
l_p\left(f_{k_1}\left(v_{s(1)},\ldots,v_{s(k_1)}\right),
\ldots,
f_{k_p}\left((v_{s(n-k_p+1)},\ldots,v_{s(n)}\right)\right)
\end{multline*}
are satisfied for any $n\in\N$ and any vectors $v_1,\ldots,v_n\in V$.

The morphism is called \textbf{strict}, 
if in addition $f_k=0$ for all $k \geq 2$,that is, if the morphism is a
single map, that commutes with all brackets. 
\end{definition}
\end{appendix}

\end{document}